\documentclass[12pt]{article}
\usepackage{amsfonts, amsmath, amssymb, latexsym, eucal, amscd} 

\usepackage[dvips]{pict2e}

\usepackage{amsthm}
\usepackage{amscd}

\usepackage[dvipdfmx]{graphics}
\usepackage{cite}

\newtheorem{definition}{Definition}[section]
\newtheorem{theorem}[definition]{Theorem}
\newtheorem{lemma}[definition]{Lemma}
\newtheorem{corollary}[definition]{Corollary}

\newtheorem{example}[definition]{Example}

\newtheorem{note}[definition]{Note}

\newtheorem{proposition}[definition]{Proposition}

\begin{document} 

\title{\bf The $S_3$-symmetric $q$-Onsager algebra \\
and its Lusztig automorphisms
}
\author{
Paul Terwilliger
\\
\\
{\it Dedicated to Tom Koornwinder on his 80th birthday}
}
\date{}
\maketitle
\begin{abstract}
The $q$-Onsager algebra $O_q$ is defined by two generators and two relations, called the $q$-Dolan/Grady relations. 
In 2019, Baseilhac and Kolb introduced two  automorphisms of $O_q$, now  called the Lusztig automorphisms.
Recently, we introduced a generalization of $O_q$ called the $S_3$-symmetric $q$-Onsager algebra $\mathbb O_q$. 
The algebra $\mathbb O_q$ has six distinguished generators, said to be standard.
The standard $\mathbb O_q$-generators can be identified with the vertices of a regular hexagon, such that nonadjacent generators commute
and adjacent generators satisfy the $q$-Dolan/Grady relations.
In the present paper we do the following: (i) for each standard $\mathbb O_q$-generator we construct an automorphism of $\mathbb O_q$ called a Lusztig
automorphism; (ii) we describe how the six Lusztig automorphisms of $\mathbb O_q$ are related to each other; (iii) we describe what happens if a
 finite-dimensional irreducible $\mathbb O_q$-module is twisted by a Lusztig automorphism; (iv) we give a detailed example involving an irreducible 
 $\mathbb O_q$-module with dimension 5.
 \bigskip

 \end{abstract}
 
\section{Introduction} The $q$-Onsager algebra originated in Algebraic Combinatorics, in the study of $Q$-polynomial distance-regular graphs. The origin story
is summarized as follows.
In \cite[Lemma~5.4]{tSub3} it was shown that for  a $Q$-polynomial distance-regular graph $\Gamma$, the adjacency matrix $A$ and a certain diagonal matrix $A^*$ satisfy two relations of the form
\begin{align}
\lbrack A, A^2 A^* - \beta A A^* A + A^* A^2 - \gamma(A A^*+A^*A) - \varrho A^* \rbrack &=0, \label{eq:Rel1} \\
\lbrack A^*, A^{*2} A - \beta A^* A A^* + A A^{*2} - \gamma^*(A^*A+AA^*) - \varrho^* A \rbrack &=0, \label{eq:Rel2}
\end{align}
\noindent where $\lbrack R,S\rbrack=RS-SR$. The scalar parameters $\beta, \gamma, \gamma^*, \varrho, \varrho^*$ depend on $\Gamma$.
The  relations  \eqref{eq:Rel1}, \eqref{eq:Rel2} are now called the tridiagonal relations 
 \cite[Section~3]{qSerre}. In \cite{qSerre}  the following algebra was introduced.
For any scalars $\beta, \gamma, \gamma^*, \varrho, \varrho^*$ the tridiagonal algebra $T=T(\beta, \gamma, \gamma^*, \varrho, \varrho^*)$ is defined by two
generators $A, A^*$ subject to the tridiagonal relations \eqref{eq:Rel1}, \eqref{eq:Rel2}.
 The finite-dimensional irreducible $T$-modules are described in \cite[Theorems~3.7,~3.10]{qSerre}, using
the concept of a tridiagonal pair of linear transformations \cite{someAlg}.
Given a tridiagonal algebra $T(\beta, \gamma, \gamma^*, \varrho, \varrho^*)$ one can change the parameters by adjusting the generators with
a change of variables $A \mapsto rA+s$, $A^*\mapsto r^* A^*+s^*$, where $r,r^*, s, s^*$ are scalars  with $r, r^*$ nonzero. The precise effect on the parameters is explained in \cite[Section~4]{qSerre}.
Via this change of parameters, one can potentially express the given tridiagonal algebra in a more attractive form. The following two forms are of interest.

\begin{example} {\rm (See \cite[Example~3.2, Remark~3.8]{qSerre}.) } \rm
Assume the ground field has characteristic 0. For
\begin{align*}
\beta=2, \qquad \quad \gamma= \gamma^*=0, \qquad \quad \varrho =4, \qquad\quad  \varrho^* =4
\end{align*}
the tridiagonal relations become the Dolan/Grady relations
\begin{align*}
\lbrack A, \lbrack A, \lbrack A, A^* \rbrack \rbrack \rbrack = 4 \lbrack A ,A^*\rbrack, \qquad \quad
\lbrack A^*, \lbrack A^*, \lbrack A^*, A \rbrack \rbrack \rbrack =4 \lbrack A^*,A\rbrack.
\end{align*}
In this case, $T(\beta, \gamma, \gamma^*, \varrho, \varrho^*)$
  becomes the  enveloping algebra $U(O)$ for the  Onsager Lie algebra $O$. 
 \end{example}
\noindent The Onsager Lie algebra $O$ was introduced in \cite{Onsager}.
Some historical comments
 about $O$ can be found in \cite[Remark~34.5]{madrid}.

\begin{example} \label{ex:Oq} {\rm (See \cite[Section~2]{qOnsager1}, \cite[Section~1]{qOnsager2}, \cite[Section~1.2]{augIto}.)} \rm
For $\beta \not=\pm 2$,
\begin{align*}
\beta=q^2 + q^{-2}, \qquad \quad \gamma=\gamma^*=0, \qquad \quad \varrho= \varrho^*=-(q^2-q^{-2})^2
\end{align*}
the tridiagonal relations become the $q$-Dolan/Grady relations
\begin{align*}
A^3 A^*- \lbrack 3 \rbrack_q A^2 A^* A + \lbrack 3 \rbrack_q A A^* A^2 - A^* A^3 = (q^2-q^{-2})^2 (A^*A-AA^*), \\
A^{*3} A- \lbrack 3 \rbrack_q A^{*2} A A^*+ \lbrack 3 \rbrack_q A^* A A^{*2} - A A^{*3} = (q^2-q^{-2})^2 (AA^*-A^*A),
\end{align*}
where $\lbrack 3 \rbrack_q = q^2+q^{-2}+1$.
In this case, we call $T(\beta, \gamma, \gamma^*, \varrho, \varrho^*)$  the $q$-Onsager algebra $O_q$.
\end{example}
\noindent  The algebra $O_q$ has applications to tridiagonal pairs of $q$-Racah type
 \cite[Chapter~6]{bbit}, \cite{itoTD, augIto, TDqRacah, TDqtet} and boundary integrable systems \cite{qOnsager1, qOnsager2,bas8,basXXZ,
basBel, BK05, bas4, basKoi, basnc, lemarthe}.
The algebra $O_q$ can be viewed as a coideal subalgebra of the quantized enveloping algebra $U_q(\widehat{\mathfrak{sl}}_2)$, see \cite{bas8, basXXZ, kolb2}.
The algebra $O_q$ is the simplest example of a quantum symmetric pair coideal subalgebra of affine type \cite[Example~7.6]{kolb2}.
A Drinfeld type presentation of $O_q$ is obtained in \cite{wang2}, and this is used in \cite{iHall} to realize $O_q$ as an $\iota$Hall algebra of the projective line.
 The alternating central extension of $O_q$ is discussed in 
\cite{BK05, basnc, goff, congTer, ACETer,comTer,TerACE}.
\medskip

\noindent The present paper is motivated by two recent developments concerning $O_q$. In \cite{kolb}, Baseilhac and Kolb introduce two 
 automorphisms of $O_q$, now  called the Lusztig automorphisms. These automorphisms are discussed in \cite{LusTer}, \cite{twist} and Section 3 below.
In \cite{S3}, we introduced a generalization of the tridiagonal algebra $T(\beta, \gamma, \gamma^*, \varrho, \varrho^*)$
called the $S_3$-symmetric tridiagonal algebra $\mathbb T(\beta, \gamma, \gamma^*, \varrho, \varrho^*)$.
This algebra has six distinguished generators, said to be standard.
The standard generators can be identified with the vertices of a regular hexagon, such that nonadjacent generators commute
and adjacent generators satisfy the tridiagonal relations. Given a $Q$-polynomial distance-regular graph $\Gamma$, in \cite[Theorem~5.4]{S3}
we turned the tensor power $V^{\otimes 3}$ of the standard module $V$  into a module for an $S_3$-symmetric tridiagonal algebra.
\medskip

\noindent 
In the present paper we consider the $S_3$-symmetric $q$-Onsager algebra $\mathbb O_q$.
Our results are summarized as follows: (i) for each standard $\mathbb O_q$-generator we construct an automorphism of $\mathbb O_q$ called a Lusztig
automorphism; (ii) we describe how the six Lusztig automorphisms of $\mathbb O_q$ are related to each other; (iii) we describe what happens if a
 finite-dimensional irreducible $\mathbb O_q$-module is twisted by a Lusztig automorphism; (iv) we give a detailed example involving an irreducible 
 $\mathbb O_q$-module with dimension 5. 
 \medskip
 
 \noindent
 Our main results are Theorems \ref{thm:main1}, \ref{thm:SixLA},  \ref{thm:LPsi}  and Proposition \ref{prop:Lint}.
 \medskip
 
 \noindent This paper is organized as follows. Section 2 contains some preliminaries.
 In Section 3, we review the $q$-Onsager algebra $O_q$ and its Lusztig automorphisms.
 In Section 4, we discuss the $S_3$-symmetric $q$-Onsager algebra $\mathbb O_q$ and how it is related to $O_q$.
 In Section 5, we introduce the Lusztig automorphisms of $\mathbb O_q$ and describe how they are related to each other.
 In Section 6, we describe what happens when a finite-dimensional irreducible $\mathbb O_q$-module is twisted via a Lusztig automorphism.
 In Section 7, we give a detailed example involving an irreducible 
 $\mathbb O_q$-module with dimension 5. 
\medskip

\noindent For notational convenience, for the rest of the paper we will use the symbol $B$ instead of $A^*$.

\section{Preliminaries} We now begin our formal argument.
For the rest of the paper,
the following notational conventions  are in effect.
Recall the natural numbers $\mathbb N=\lbrace 0,1,2,\ldots \rbrace$.
Let $\mathbb F$ denote a field. Every vector space and tensor product we discuss, is understood to be over $\mathbb F$.
Every algebra we discuss, is understood to be associative, over $\mathbb F$, and have a multiplicative identity.
A subalgebra has the same multiplicative identity as the parent algebra. For an algebra $\mathcal A$,
by an  automorphism of $\mathcal A$ we mean an algebra isomorphism $\mathcal A \to \mathcal A$.
The automorphism group  ${\rm Aut}(\mathcal A)$ consists of the automorphisms of $\mathcal A$; the group operation is composition.
For an integer $n\geq 1$ let $\lambda_1, \lambda_2, \ldots, \lambda_n$ denote commuting indeterminates.
The algebra $\mathbb F\lbrack \lambda_1, \lambda_2, \ldots, \lambda_n\rbrack$ consists of the polynomials in
$\lambda_1, \lambda_2, \ldots, \lambda_n$ that have all coefficients in $\mathbb F$.
Let $V$ denote a nonzero vector space with finite dimension. The algebra ${\rm End}(V)$ consists of the $\mathbb F$-linear
maps from $V$ to $V$. For $F \in {\rm End}(V)$, we say that $F$ is {\it diagonalizable} whenever $V$ is spanned by
the eigenspaces of $F$.
 Assume that $F$ is diagonalizable, and let $\lbrace V_i \rbrace_{i=0}^d$ 
denote an ordering of the eigenspaces of $F$. The sum $V=\sum_{i=0}^d V_i$ is direct.
For $0 \leq i \leq d$ let $\theta_i$ denote the eigenvalue of $F$ for $V_i$. By construction, the scalars $\lbrace \theta_i \rbrace_{i=0}^d$ are
contained in $\mathbb F$ and mutually distinct. 
For $0 \leq i \leq d$ define $E_i \in {\rm End}(V)$ such that
$(E_i-I) V_i=0$ and $E_iV_j=0$ if $j \not=i$ $(0 \leq j \leq d)$. Thus $E_i$ is the projection from $V$ onto $V_i$. We call $E_i$ the {\it primitive idempotent} of $F$ associated with $V_i$ (or $\theta_i$).
By linear algebra 
(i) $I = \sum_{i=0}^d E_i$; 
(ii) $E_i E_j = \delta_{i,j} E_i$ $ (0 \leq i,j\leq d)$;
(iii) $F = \sum_{i=0}^d \theta_i E_i$;
(iv) $FE_i = \theta_i E_i = E_iF$ $(0 \leq i \leq d)$;
(v) $V_i = E_iV$ $ (0 \leq i \leq d)$.
Also by linear algebra,
\begin{align*}
  E_i=\prod_{\stackrel{0 \leq j \leq d}{j \neq i}}
          \frac{F-\theta_jI}{\theta_i-\theta_j} \qquad \qquad (0 \leq i \leq d).
\end{align*}
\noindent Fix a nonzero $q \in \mathbb F$ that is not a root of unity. Define
\begin{align*}
\lbrack n \rbrack_q = \frac{ q^n-q^{-n}}{q-q^{-1}} \qquad \qquad n \in \mathbb N.
\end{align*}
\noindent For elements $R,S$ in any algebra, their commutator is $\lbrack R,S\rbrack=RS-SR$.

\section{The  $q$-Onsager algebra $O_q$}
\noindent  In this section, we review the $q$-Onsager algebra $O_q$ and its Lusztig automorphisms.
As we mentioned in Example \ref{ex:Oq},
  the algebra
$O_q$  is defined by two
generators $A, B$ and  two relations
\begin{align*}
A^3 B- \lbrack 3 \rbrack_q A^2 B A + \lbrack 3 \rbrack_q A B A^2 - B A^3 = (q^2-q^{-2})^2 (BA-AB), \\
B^3 A- \lbrack 3 \rbrack_q B^2 A B+ \lbrack 3 \rbrack_q B A B^2 - A B^3 = (q^2-q^{-2})^2 (AB-BA).
\end{align*}
The above relations are called the $q$-Dolan/Grady relations. 
\medskip

\noindent The vector space $O_q$ is infinite-dimensional. Moreover the elements $A^i B^j$ $(i,j \in \mathbb N)$
are linearly independent in $O_q$, because there  exists an algebra
homomorphism $O_q \to \mathbb F\lbrack \lambda_1, \lambda_2\rbrack$ that sends $A\mapsto \lambda_1$
and $B\mapsto \lambda_2$. 
\medskip

\noindent Next, we discuss the automorphisms of $O_q$. By construction,  $O_q$ has an automorphism that sends $A \leftrightarrow B$. 
 In \cite{kolb}, Pascal Baseilhac and Stefan Kolb introduced two automorphisms of $O_q$ (here denoted by $L$ and $L^*$) that act as follows:
\begin{align*}
&L(A) = A, \qquad \qquad L(B) = B+ \frac{q A^2 B - (q+q^{-1}) A B A + q^{-1} B A^2}{(q-q^{-1})(q^2-q^{-2})},
\\
&L^*(B) = B, \qquad \qquad L^*(A) = A+ \frac{q B^2 A - (q+q^{-1}) BAB  + q^{-1} A B^2}{(q-q^{-1})(q^2-q^{-2})}.
\end{align*}
\noindent By \cite{kolb}, the inverses of $L$ and $L^*$ act as follows:
\begin{align*}
&L^{-1}(A) = A, \qquad \qquad L^{-1}(B) = B+ \frac{q^{-1} A^2 B - (q+q^{-1}) A B A + q B A^2}{(q-q^{-1})(q^2-q^{-2})},
\\
&(L^*)^{-1}(B) = B, \qquad \qquad (L^*)^{-1}(A) = A+ \frac{q^{-1} B^2 A - (q+q^{-1}) BAB + q AB^2}{(q-q^{-1})(q^2-q^{-2})}.
\end{align*}
Following \cite{LusTer} we call  $L$ (resp. $L^*$) the {\it Lusztig automorphism of $O_q$ associated with $A$ (resp. $B$)}. 
\medskip

\noindent Let $G$ denote the subgroup of ${\rm Aut}(O_q)$ generated by $L^{\pm 1}$, $(L^*)^{\pm 1}$. By \cite[Section~2]{qOnsUAW}, $G$ is freely
generated by $L^{\pm 1}, (L^*)^{\pm 1}$.

\begin{note} \label{note:global} \rm In \cite{qOnsUAW} there is a global assumption that $\mathbb F$ is algebraically closed with characteristic zero. However,
this assumption is not used in \cite[Section~2]{qOnsUAW} to prove that $G$ is  freely
generated by $L^{\pm 1}, (L^*)^{\pm 1}$.
\end{note}

\section{The $S_3$-symmetric $q$-Onsager algebra}
\noindent Let $S_3$ denote the symmetric group on the set $\lbrace 1,2,3\rbrace$. In \cite{S3}  we introduced a generalization of $O_q$ called the
$S_3$-symmetric $q$-Onsager algebra (here denoted by $\mathbb O_q$). We now recall the definition of $\mathbb O_q$.

\begin{definition} \label{def:bbO} \rm (See \cite[Lemma~3.5, Definition~4.1]{S3}.)
The algebra $\mathbb O_q$
is defined by generators 
\begin{align}
A_i, \quad B_i \qquad \quad i \in \lbrace 1,2,3\rbrace \label{eq:OGen}
\end{align}
and the following relations.
\begin{enumerate}
\item[\rm (i)] For distinct $i,j \in \lbrace 1,2,3\rbrace$,
\begin{align*}
\lbrack A_i, A_j \rbrack=0, \qquad \qquad \lbrack B_i, B_j \rbrack=0.
\end{align*}
\item[\rm (ii)] For $i \in \lbrace 1,2,3\rbrace$,
\begin{align*}
\lbrack A_i, B_i \rbrack=0.
\end{align*}
\item[\rm (iii)] For distinct $i,j \in \lbrace 1,2,3 \rbrace$,
\begin{align*}
A_i^3 B_j- \lbrack 3 \rbrack_q A_i^2 B_j A_i + \lbrack 3 \rbrack_q A_i B_j A_i^2 - B_j A_i^3 = (q^2-q^{-2})^2 (B_jA_i-A_iB_j), \\
B_j^3 A_i- \lbrack 3 \rbrack_q B_j^2 A_i B_j+ \lbrack 3 \rbrack_q B_j A_i B_j^2 - A_i B_j^3 = (q^2-q^{-2})^2 (A_iB_j-B_jA_i).
\end{align*}
\end{enumerate}
The elements \eqref{eq:OGen} are called the {\it standard generators} for $\mathbb O_q$.
\end{definition}

\noindent  Our  presentation of $\mathbb O_q$  is described by the diagram below:
\vspace{.5cm}

\begin{center}
\begin{picture}(0,60)
\put(-100,0){\line(1,1.73){30}}
\put(-100,0){\line(1,-1.73){30}}
\put(-70,52){\line(1,0){60}}
\put(-70,-52){\line(1,0){60}}
\put(-10,-52){\line(1,1.73){30}}
\put(-10,52){\line(1,-1.73){30}}

\put(-102,-3){$\bullet$}
\put(-72.5,49){$\bullet$}
\put(-72.5,-55){$\bullet$}

\put(16.5,-3){$\bullet$}
\put(-13.5,49){$\bullet$}
\put(-13.5,-55){$\bullet$}

\put(-120,-3){$A_1$}
\put(-11,60){$A_2$}
\put(-11,-66){$A_3$}

\put(26,-3){$B_1$}
\put(-82,-66){$B_2$}
\put(-82,60){$B_3$}

\end{picture}
\end{center}
\vspace{2.5cm}

\hspace{2.2cm}
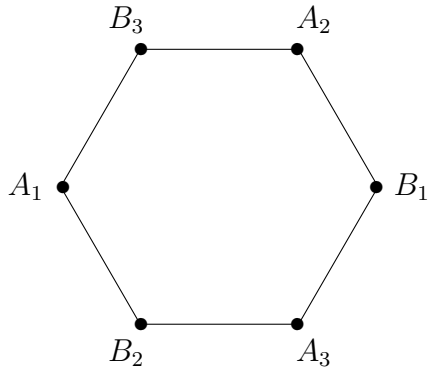
 Fig. 1.  Nonadjacent generators commute. \\
${} \qquad \quad \quad \;\;$ Adjacent generators satisfy the $q$-Dolan/Grady relations.
\vspace{.5cm}

\noindent  Next, we describe how $O_q$ and $\mathbb O_q$ are related. By \cite[Lemmas~4.2, 4.4]{S3},
for distinct $i,j \in \lbrace 1,2,3\rbrace $ there exists an injective algebra homomorphism $O_q \to \mathbb O_q$
that sends $A \mapsto A_i$ and $B\mapsto B_j$.
The next three lemmas are of a similar nature. 

\begin{lemma}  \label{lem:OO} For distinct $i,j \in \lbrace 1,2,3\rbrace$ there exists an algebra homomorphism $O_q \otimes O_q \to \mathbb O_q$
that sends
\begin{align*}
A \otimes 1 \mapsto A_i, \qquad \quad
B \otimes 1 \mapsto B_j, \qquad \quad
1 \otimes A \mapsto A_j, \qquad \quad
1 \otimes B \mapsto B_i.
\end{align*}
\noindent This homomorphism is injective.
\end{lemma}
\begin{proof}  The elements $A_i, B_j$ satisfy the $q$-Dolan/Grady relations.
The elements $A_j, B_i$ satisfy the $q$-Dolan/Grady relations.
Each of $A_i, B_j$ commutes with each of $A_j,B_i$. By these comments, the above homomorphism $O_q \otimes O_q \to \mathbb O_q$
 exists. We show that this homomorphism is injective.
Let $k \in \lbrace 1,2,3\rbrace \backslash \lbrace i,j\rbrace$.
By the form of the relations in Definition  \ref{def:bbO},
there exists an algebra homomorphism $\mathbb O_q \to O_q \otimes O_q$ that sends 
\begin{align*}
&A_i \mapsto A \otimes 1, \qquad \qquad A_j \mapsto 1 \otimes A, \qquad \qquad A_k\mapsto 0, \\
& B_i \mapsto 1 \otimes B, \qquad \qquad B_j \mapsto B \otimes 1, \qquad \qquad B_k \mapsto 0.
\end{align*}
The composition 
$O_q\otimes O_q \to \mathbb O_q \to O_q \otimes O_q$ fixes each of
\begin{align*}
A \otimes 1, \qquad \quad
B \otimes 1, \qquad \quad
1 \otimes A, \qquad \quad
1 \otimes B.
\end{align*}
Therefore, this composition is the identity map on $O_q \otimes O_q$. It follows that the homomorphism $O_q \otimes O_q \to \mathbb O_q$ is injective.
\end{proof}

\begin{lemma}\label{lem:OOO} There exists an algebra homomorphism $\mathbb O_q \to O_q \otimes O_q \otimes O_q$ that sends
\begin{align*}
&A_1 \mapsto A \otimes 1 \otimes 1, \qquad \qquad 
A_2 \mapsto 1 \otimes A \otimes 1, \qquad \qquad 
A_3  \mapsto  1\otimes 1 \otimes A, \\
& B_1 \mapsto 1 \otimes B \otimes 1, \qquad \qquad
B_2 \mapsto 1 \otimes 1 \otimes B, \qquad \qquad
B_3 \mapsto  B \otimes 1\otimes 1.
\end{align*}
This homomorphism is surjective.
\end{lemma} 
\begin{proof} The above six tensor products satisfy the defining relations for $\mathbb O_q$ given in Definition \ref{def:bbO}. Therefore the
homomorphism exists. The homomorphism is surjective  because $A, B$ generate $O_q$.
\end{proof}

\begin{lemma}\label{lem:OOO2} There exists an algebra homomorphism $\mathbb O_q \to O_q \otimes O_q \otimes O_q$ that sends
\begin{align*}
&A_1 \mapsto A \otimes 1 \otimes 1, \qquad \qquad 
A_2 \mapsto 1 \otimes A \otimes 1, \qquad \qquad 
A_3  \mapsto  1\otimes 1 \otimes A, \\
& B_1 \mapsto 1 \otimes 1 \otimes B, \qquad \qquad
B_2 \mapsto B \otimes 1 \otimes 1, \qquad \qquad
B_3 \mapsto  1 \otimes B\otimes 1.
\end{align*}
This homomorphism is surjective.
\end{lemma} 
\begin{proof} Similar to the proof of Lemma \ref{lem:OOO}.
\end{proof}

\begin{corollary} \label{cor:AAABBB} {\rm (See \cite[Lemma~4.8]{S3}.)} The elements $1, A_1,A_2,A_3,B_1,B_2, B_3$ are linearly independent.
\end{corollary}
\begin{proof} The elements $1, A,B$ are linearly independent in $O_q$, so the following elements are linearly
independent in  $O_q \otimes O_q \otimes O_q$:
\begin{align*}
1 \otimes 1 \otimes 1,& \qquad \qquad A \otimes 1 \otimes 1, \qquad \qquad 
1 \otimes A \otimes 1, \qquad \qquad 
 1\otimes 1 \otimes A, \\
&  B \otimes 1\otimes 1, \qquad \qquad 1 \otimes B \otimes 1, \qquad \qquad
 1 \otimes 1 \otimes B.
\end{align*}
The result follows from this and Lemma \ref{lem:OOO} or \ref{lem:OOO2}.
\end{proof}
\noindent Next, we consider the automorphisms of $\mathbb O_q$. 
 The dihedral group $D_6$ is the group of symmetries of a regular hexagon. The group $D_6$ has 12 elements.
 The action of $D_6$ on the vertices of the regular hexagon induces a group homomorphism $D_6 \to {\rm Aut}(\mathbb O_q)$.
 This homomorphism is injective in view of Corollary \ref{cor:AAABBB}. Thus the group ${\rm Aut}(\mathbb O_q)$ contains a subgroup
 isomorphic to $D_6$. In the next section, we consider some additional elements in ${\rm Aut}(\mathbb O_q)$.

\section{The Lusztig automorphisms of  $\mathbb O_q$}
\noindent In this section, we introduce the Lusztig automorphisms of $\mathbb O_q$.

 \begin{theorem}\label{thm:main1} For $i \in \lbrace 1,2,3\rbrace$ there exist  $L_i, L^*_i \in {\rm Aut}(\mathbb O_q)$ that do the following:
 \begin{align*}
 &L_i(A_1) = A_1, \qquad   \quad  L_i(A_2) = A_2, \qquad \quad L_i(A_3) = A_3,    \quad     \qquad L_i(B_i) = B_i,\\
 &L^*_i(B_1) = B_1, \qquad \quad     L^*_i(B_2) = B_2, \qquad \quad  L^*_i(B_3) = B_3, \qquad \quad      L^*_i(A_i) = A_i
 \end{align*}
 \noindent and for $j \in \lbrace 1,2,3\rbrace\backslash \lbrace i \rbrace$,
 \begin{align*}
 & L_i(B_j) = B_j+ \frac{q A_i^2 B_j - (q+q^{-1}) A_i B_j A_i + q^{-1} B_j A_i^2}{(q-q^{-1})(q^2-q^{-2})},\\
 & L_i^{-1}(B_j) = B_j+ \frac{q^{-1} A_i^2 B_j - (q+q^{-1}) A_i B_j A_i+ q B_j A_i^2}{(q-q^{-1})(q^2-q^{-2})}, \\
  & L^*_i(A_j) = A_j+ \frac{q B_i^2 A_j - (q+q^{-1}) B_i A_j B_i + q^{-1} A_j B_i^2}{(q-q^{-1})(q^2-q^{-2})},\\
 & (L^*_i)^{-1}(A_j) = A_j+ \frac{q^{-1} B_i^2 A_j - (q+q^{-1}) B_i A_j B_i+ q A_j B_i^2}{(q-q^{-1})(q^2-q^{-2})}.
\end{align*}

 \end{theorem}
 \begin{proof} We first verify our assertions about $L_i$. We invoke \cite{LusTer}. In \cite[Section~3]{LusTer} we encounter an algebra $\mathcal A$ and a distinguished  element $A \in \mathcal A$; we take $\mathcal A$ to be $\mathbb O_q$ and the distinguished element to be $A_i$. 
  Let $X$ denote a standard generator for $\mathbb O_q$.
 By Definition \ref{def:bbO}, either $X$ commutes with $A_i$ or
 \begin{align} \label{eq:AAAX}
 A_i^3 X - \lbrack 3 \rbrack_q  A_i^2 X A_i + \lbrack 3 \rbrack_q A_i X A_i^2 - X A_i^3 =
 (q^2-q^{-2})^2 (X A_i - A_i X).
 \end{align}
 By these comments and \cite[Sections~4,~5]{LusTer}  there exists $L_i \in {\rm Aut}(\mathbb O_q)$ that does the following:
(i) for all $X \in \mathbb O_q$ that commutes with $A_i$, we have  $L_i(X)=X$; (ii) for all $X \in \mathbb O_q$ that satisfies \eqref{eq:AAAX}, we have
 \begin{align*}
 &L_i(X) = X+ \frac{q A_i^2 X - (q+q^{-1}) A_i X A_i + q^{-1} X A_i^2}{(q-q^{-1})(q^2-q^{-2})}, \\
 & L_i^{-1}(X) = X+ \frac{q^{-1} A_i^2 X - (q+q^{-1}) A_i X A_i+ q X A_i^2}{(q-q^{-1})(q^2-q^{-2})}.
 \end{align*}
 The element $A_i$ commutes with each of $A_1, A_2, A_3, B_i$. 
 We have verified our assertions about $L_i$. Our assertions about $L^*_i$ are similarly verified.
 \end{proof}
 
\begin{definition} \label{def:LA} \rm Referring to Theorem \ref{thm:main1}, we call $L_i$ (resp. $L^*_i$) 
the {\it Lusztig automorphism of $\mathbb O_q$ associated with $A_i$ (resp. $B_i$)}.
\end{definition}
\noindent We record a feature of the Lusztig automorphisms that appeared in the proof of 
Theorem \ref{thm:main1}. 

\begin{lemma} \label{lem:FeatureL} The following hold for all $i \in \lbrace 1,2,3\rbrace$ and $X \in \mathbb O_q$.
\begin{enumerate}
\item[\rm (i)] Assume that $X$ commutes with $A_i$. The $L_i(X)=X$.
\item[\rm (ii)]  Assume that
\begin{align*} 
 A_i^3 X - \lbrack 3 \rbrack_q  A_i^2 X A_i + \lbrack 3 \rbrack_q A_i X A_i^2 - X A_i^3 =
 (q^2-q^{-2})^2 (X A_i - A_i X).
 \end{align*}
 Then 
  \begin{align*}
 &L_i(X) = X+ \frac{q A_i^2 X - (q+q^{-1}) A_i X A_i + q^{-1} X A_i^2}{(q-q^{-1})(q^2-q^{-2})}, \\
 & L_i^{-1}(X) = X+ \frac{q^{-1} A_i^2 X - (q+q^{-1}) A_i X A_i+ q X A_i^2}{(q-q^{-1})(q^2-q^{-2})}.
 \end{align*}
\end{enumerate}
\end{lemma}
\begin{proof} Immediate from the proof of Theorem \ref{thm:main1}.
\end{proof}

\begin{lemma} \label{lem:FeatureLs} The following hold for all $i \in \lbrace 1,2,3\rbrace$ and $X \in \mathbb O_q$.
\begin{enumerate}
\item[\rm (i)] Assume that $X$ commutes with $B_i$. The $L^*_i(X)=X$.
\item[\rm (ii)]  Assume that
\begin{align*} 
 B_i^3 X - \lbrack 3 \rbrack_q  B_i^2 X B_i + \lbrack 3 \rbrack_q B_i X B_i^2 - X B_i^3 =
 (q^2-q^{-2})^2 (X B_i - B_i X).
 \end{align*}
 Then 
  \begin{align*}
 &L^*_i(X) = X+ \frac{q B_i^2 X - (q+q^{-1}) B_i X B_i + q^{-1} X B_i^2}{(q-q^{-1})(q^2-q^{-2})}, \\
 & (L^*_i)^{-1}(X) = X+ \frac{q^{-1} B_i^2 X - (q+q^{-1}) B_i X B_i+ q X B_i^2}{(q-q^{-1})(q^2-q^{-2})}.
 \end{align*}
\end{enumerate}
\end{lemma}
\begin{proof} Immediate from the proof of Theorem \ref{thm:main1}.
\end{proof}

\noindent Next, we describe how the Lusztig automorphisms of $O_q$ are related to the Lusztig automorphisms of $\mathbb O_q$.

\begin{lemma} \label{lem:Injera}
Pick distinct $i,j \in \lbrace 1,2,3\rbrace$ and consider the injective algebra homomorphism $\sigma_{i,j}: O_q \to \mathbb O_q$ that sends
$A \mapsto A_i$ and $B \mapsto B_j$. Then the following diagrams commute:
\begin{equation*}
{\begin{CD}
O_q @>\sigma_{i,j} >>
               \mathbb O_q
              \\
         @V L VV                   @VV L_i V \\
         O_q @>>\sigma_{i,j} >
                                  \mathbb O_q
                        \end{CD}} 
                        \qquad \qquad \quad
    {\begin{CD}
O_q @>\sigma_{i,j} >>
               \mathbb O_q
              \\
         @V L^* VV                   @VV L^*_jV \\
         O_q @>>\sigma_{i,j} >
                                  \mathbb O_q
                        \end{CD}} 
\end{equation*}
\end{lemma}
\begin{proof} Chase the $O_q$-generators $A,B$ around each diagram, using the definition of $L, L^*$ and
Theorem \ref{thm:main1}. 
\end{proof}

\noindent Next, we describe how the Lusztig automorphisms of $\mathbb O_q$ are related to each other.

\begin{theorem}\label{thm:SixLA} The Lusztig automorphisms of $\mathbb O_q$ are related as follows.
\begin{enumerate}
\item[\rm (i)] For distinct $i,j \in \lbrace 1,2,3\rbrace$,
\begin{align*}
L_i L_j = L_j L_i, \qquad \qquad L^*_i L^*_j = L^*_j L^*_i.
\end{align*}
\item[\rm (ii)] For $i \in \lbrace 1,2,3\rbrace$,
\begin{align*}
L_i L^*_i = L^*_i L_i.
\end{align*}
\item[\rm (iii)] For distinct $i,j \in \lbrace 1,2,3 \rbrace$, 
\begin{center}$L^{\pm 1}_i, (L^*_j)^{\pm 1}$ freely generate a subgroup of ${\rm Aut}(\mathbb O_q)$. \end{center}
\end{enumerate}
\end{theorem}
\begin{proof} (i) We first show that $L_i L_j = L_jL_i$.
To do this, we show that 
 \begin{align} \label{eq:LLX}
           L_i L_j(X) = L_jL_i(X)
\end{align}           
            for every standard generator $X$  of $\mathbb O_q$.
First assume that $X$ is one of $A_1, A_2, A_3$. Then \eqref{eq:LLX} holds because $L_i(X)=X$ and $L_j(X)=X$.
Next assume that $X=B_i$.
Let us compute $L_iL_j(B_i)$.
 By Theorem \ref{thm:main1},
\begin{align} \label{eq:LJB}
 L_j(B_i) = B_i+ \frac{q A_j^2 B_i - (q+q^{-1}) A_jB_i A_j + q^{-1} B_i A_j^2}{(q-q^{-1})(q^2-q^{-2})}.
\end{align}
We  apply $L_i$ to each side of \eqref{eq:LJB}, and evaluate the result using $L_i(A_j)=A_j$ and $L_i(B_i)=B_i$; this yields
$L_iL_j(B_i)=L_j(B_i)$. We just mentioned that $L_i(B_i)=B_i$, so $L_jL_i(B_i) = L_j(B_i)$. By these comments,
 \eqref{eq:LLX} holds for $X=B_i$.
One similarly shows that \eqref{eq:LLX} holds for $X=B_j$.
We now show that   \eqref{eq:LLX} holds for $X=B_k$ with $k \in \lbrace 1,2,3\rbrace\backslash \lbrace i,j \rbrace$.
Let us compute $L_iL_j(B_k)$. By Theorem \ref{thm:main1},
\begin{align} \label{eq:LB}
 L_j(B_k) = B_k+ \frac{q A_j^2 B_k - (q+q^{-1}) A_jB_k A_j + q^{-1} B_k A_j^2}{(q-q^{-1})(q^2-q^{-2})}.
\end{align}
We  apply $L_i$ to each side of \eqref{eq:LB}, and evaluate the result using $L_i(A_j)=A_j$; this yields
\begin{align} \label{eq:LB2}
L_i L_j(B_k) = L_i(B_k)+ \frac{q A_j^2 L_i(B_k) - (q+q^{-1}) A_jL_i(B_k) A_j + q^{-1} L_i(B_k) A_j^2}{(q-q^{-1})(q^2-q^{-2})}.
\end{align}
We now  compute $L_jL_i(B_k)$. By Definition  \ref{def:bbO}(iii),
\begin{align}\label{eq:AAB}
 A_j^3 B_k - \lbrack 3 \rbrack_q  A_j^2 B_k A_j + \lbrack 3 \rbrack_q A_j B_k A_j^2 - B_k A_j^3 =
 (q^2-q^{-2})^2 (B_k A_j- A_j B_k).
\end{align}
We apply $L_i$ to each side of \eqref{eq:AAB}, and evaluate the result using $L_i(A_j)=A_j$; this yields
\begin{align*}
 A_j^3 L_i(B_k) &- \lbrack 3 \rbrack_q  A_j^2 L_i(B_k) A_j + \lbrack 3 \rbrack_q A_j L_i(B_k) A_j^2 - L_i(B_k) A_j^3 \\
& =
 (q^2-q^{-2})^2 \bigl(L_i(B_k) A_j- A_j L_i(B_k) \bigr).
\end{align*}
By this and Lemma \ref{lem:FeatureL}(ii) with $X=L_i(B_k)$,
\begin{align} \label{eq:LB3}
L_j L_i(B_k) = L_i(B_k)+ \frac{q A_j^2 L_i(B_k) - (q+q^{-1}) A_jL_i(B_k) A_j + q^{-1} L_i(B_k) A_j^2}{(q-q^{-1})(q^2-q^{-2})}.
\end{align}
Comparing \eqref{eq:LB2}, \eqref{eq:LB3} we see that \eqref{eq:LLX} holds for $X=B_k$.
We have shown that \eqref{eq:LLX} holds for every standard generator $X$ of $\mathbb O_q$.
Therefore $L_iL_j=L_jL_i$.
One similarly shows that  $L^*_iL^*_j=L^*_jL^*_i$. \\
\noindent (ii) It suffices to show that 
\begin{align} \label{eq:LLSX}
L_i L^*_i (X)=L^*_i L_i(X)
\end{align}
for every standard generator $X$ of $\mathbb O_q$.
First assume that $X$ is one of $A_i, B_i$. Then \eqref{eq:LLSX} holds because $L_i(X)=X$ and $L^*_i(X)=X$.
Next assume that  $X=A_j$ with  $j \in \lbrace 1,2,3\rbrace\backslash \lbrace i \rbrace$.
Let us compute $L_i L^*_i (A_j)$. By Theorem \ref{thm:main1},
\begin{align} \label{eq:LLA}
 L^*_i(A_j) = A_j+ \frac{q B_i^2 A_j - (q+q^{-1}) B_i A_j B_i + q^{-1} A_j B_i^2}{(q-q^{-1})(q^2-q^{-2})}.
\end{align}
We apply $L_i$ to each side of \eqref{eq:LLA}, and evaluate the result using $L_i(A_j)=A_j$ and $L_i(B_i)=B_i$; this yields
$L_i L^*_i(A_j)  = L^*_i(A_j)$. We just mentioned that $L_i(A_j)=A_j$, so 
 $L^*_i L_i(A_j) = L^*_i(A_j)$. By these comments,  \eqref{eq:LLSX} holds for $X=A_j$.
 One similarly shows that \eqref{eq:LLSX} holds for $X=B_j$.
 We have shown that \eqref{eq:LLSX} holds for every standard generator $X$ of $\mathbb O_q$. Therefore $L_iL^*_i = L^*_i L_i$.
 \\
\noindent (iii) By the discussion above Note \ref{note:global},  along with Lemma \ref{lem:Injera}.
\end{proof}

\noindent  The Lusztig automorphisms of $\mathbb O_q$  are described by the diagram below:
\vspace{.5cm}

\begin{center}
\begin{picture}(0,60)
\put(-100,0){\line(1,1.73){30}}
\put(-100,0){\line(1,-1.73){30}}
\put(-70,52){\line(1,0){60}}
\put(-70,-52){\line(1,0){60}}
\put(-10,-52){\line(1,1.73){30}}
\put(-10,52){\line(1,-1.73){30}}

\put(-102,-3){$\bullet$}
\put(-72.5,49){$\bullet$}
\put(-72.5,-55){$\bullet$}

\put(16.5,-3){$\bullet$}
\put(-13.5,49){$\bullet$}
\put(-13.5,-55){$\bullet$}

\put(-120,-3){$L_1$}
\put(-11,60){$L_2$}
\put(-11,-66){$L_3$}

\put(26,-3){$L^*_1$}
\put(-82,-66){$L^*_2$}
\put(-82,60){$L^*_3$}

\end{picture}
\end{center}
\vspace{2.5cm}

\hspace{2.2cm}
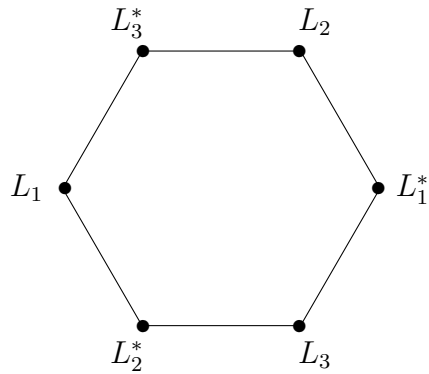
 Fig. 2.  Nonadjacent Lusztig automorphisms commute. \\
${} \qquad \quad \quad \qquad \;$ Adjacent Lusztig automorphisms and their inverses \\
${} \qquad \quad \quad \qquad \;$ freely generate a subgroup of ${\rm Aut}(\mathbb O_q)$.
\vspace{.5cm}

\section{Finite-dimensional irreducible $\mathbb O_q$-modules}

\noindent   In this section, we describe what happens
when a finite-dimensional irreducible $\mathbb O_q$-module is twisted by a Lusztig automorphism of $\mathbb O_q$. To simplify the notation, 
we will restrict our attention to  the Lusztig automorphism $L_1$. Similar results apply to the other Lusztig automorphisms.
\medskip

\noindent
Throughout this section, $V$ denotes a finite-dimensional irreducible $\mathbb O_q$-module on which $A_1$ is diagonalizable.
We will display an invertible $\Psi \in {\rm End}(V)$ such that on $V$,  $L_1(X)=\Psi^{-1} X \Psi$ for all $X \in \mathbb O_q$.
\medskip

\noindent The following definition is for notational convenience.
\begin{definition}\label{def:NSup} \rm The action of $\mathbb O_q$ on $V$ induces an algebra homomorphism $\mathbb O_q \to {\rm End}(V)$;
for $i \in \lbrace 1,2,3\rbrace$ let $A^{(i)}$ (resp. $B^{(i)}$) denote the image of $A_i $ (resp. $B_i$) under this homomorphism.  Thus on $V$,
$A_i = A^{(i)}$ and $B_i =B^{(i)}$.
\end{definition}

\noindent  We first dispense with a special case. For the moment, assume that $A^{(1)}$ is a scalar multiple of the identity map  $I$.
By the description of $L_1$ in Theorem
\ref{thm:main1}, on $V$ we have $L_1(A_i)=A_i$ and $L_1(B_i)=B_i$ for all $i \in \lbrace 1,2,3\rbrace$.
 Therefore, on $V$ we have
 $L_1(X)=X$
for all $X \in \mathbb O_q$. Taking $\Psi =I$, on $V$ we have $L_1(X)=\Psi^{-1} X \Psi$ for all $X \in \mathbb O_q$.
\medskip

\noindent For the rest of this section, we assume that $A^{(1)}$ is not a scalar multiple the identity map.

\begin{definition}\label{def:EP} \rm
Let $\lbrace E_i^{(1)} \rbrace_{i=0}^d$ denote an ordering of the primitive idempotents of $A^{(1)}$.
For $0 \leq i \leq d$ let $\theta_i=\theta_i^{(1)}$ denote the eigenvalue of $A^{(1)}$ associated with $E_i^{(1)}$. 
\end{definition}

\noindent We have some comments about Definition \ref{def:EP}. Note that $d\geq 1$,
since $A^{(1)}$ is not a scalar multiple of $I$.
As we discussed in Section 2, the eigenvalues $\lbrace \theta_i \rbrace_{i=0}^d$ are contained in $\mathbb F$ and mutually distinct.
Also,
\begin{align*}
&I = \sum_{i=0}^d E^{(1)}_i, \qquad \qquad E^{(1)}_i E^{(1)}_j = \delta_{i,j} E^{(1)}_i \qquad (0 \leq i,j\leq d), \\
&A^{(1)} = \sum_{i=0}^d \theta_i E^{(1)}_i, \qquad \qquad E^{(1)}_i A^{(1)} = \theta_i E^{(1)}_i =A^{(1)} E^{(1)}_i \qquad (0 \leq i \leq d), \\
&    E_i^{(1)} =\prod_{\stackrel{0 \leq j \leq d}{j \neq i}}
       \frac{A^{(1)}-\theta_jI}{\theta_i-\theta_j} \qquad \qquad (0 \leq i \leq d).
       \end{align*}
Moreover,
  \begin{align} \label{eq:Vsum}
  V = \sum_{i=0}^d E^{(1)}_i V \qquad \qquad \hbox{\rm (direct sum).}
  \end{align}
  For $0 \leq i \leq d$ the subspace $E^{(1)}_iV$ is the eigenspace of $A^{(1)}$ for the eigenvalue $\theta_i$.

\begin{definition} \label{def:Poly} \rm  We define a polynomial $P \in \mathbb F\lbrack \lambda_1, \lambda_2 \rbrack$ by
\begin{align*}
P(\lambda_1, \lambda_2) = \lambda_1^2 - (q^2+q^{-2}) \lambda_1 \lambda_2 + \lambda_2^2 + (q^2-q^{-2})^2.
\end{align*}
\end{definition}
\noindent Note that $P(\lambda_1, \lambda_2)= P(\lambda_2, \lambda_1)$. 
\begin{lemma} \label{lem:EBE} For $0 \leq i,j\leq d$ we have
\begin{align*}
0 = E_i^{(1)} B^{(2)} E_j^{(1)} (\theta_i - \theta_j) P(\theta_i, \theta_j), \qquad \quad  0 = E_i^{(1)} B^{(3)} E_j^{(1)}(\theta_i - \theta_j)P(\theta_i, \theta_j).
\end{align*}
\end{lemma} 
\begin{proof} We first verify the equation on the left. By Definitions \ref{def:bbO}(iii) and \ref{def:NSup},
\begin{align*}
0 &= \bigl( A^{(1)}\bigr)^3 B^{(2)}- \lbrack 3 \rbrack_q  \bigl( A^{(1)}\bigr)^2B^{(2)} A^{(1)} + 
\lbrack 3 \rbrack_q A^{(1)} B^{(2)} \bigl( A^{(1)}\bigr)^2 - B^{(2)} \bigl( A^{(1)}\bigr)^3 \\
 &\qquad \qquad \qquad + (q^2-q^{-2})^2 \bigl(A^{(1)}B^{(2)} -B^{(2)} A^{(1)}\bigr).
\end{align*}
In this equation, we multiply each term on the left by $E_i^{(1)}$ and the right by $E_j^{(1)}$. We simplify the result using
$ E^{(1)}_i A^{(1)} = \theta_i E^{(1)}_i $ and $ A^{(1)} E^{(1)}_j = \theta_j E^{(1)}_j$. This yields
\begin{align*}
0 &= E^{(1)}_i B^{(2)} E^{(1)}_j \Bigl(\theta_i^3 - \lbrack 3 \rbrack_q \theta_i^2 \theta_j +
\lbrack 3 \rbrack_q \theta_i \theta_j^2 - \theta_j^3 + (q^2-q^{-2})^2 (\theta_i - \theta_j) \Bigr) \\
&= E^{(1)}_i B^{(2)} E^{(1)}_j (\theta_i - \theta_j) \Bigl(
\theta_i^2 - (q^2+q^{-2}) \theta_i \theta_j + \theta_j^2 + (q^2-q^{-2})^2
 \Bigr) \\
 &=  E^{(1)}_i B^{(2)} E^{(1)}_j (\theta_i - \theta_j) P(\theta_i, \theta_j).
\end{align*}
\noindent We have verified the equation on the left. The equation on the right is similarly verified.
\end{proof}

\begin{definition} \label{def:diagram} \rm Using $A^{(1)}$ we  define an undirected graph $\mathcal D = \mathcal D^{(1)}$ as follows.
The vertex set of $\mathcal D$ is the set $\lbrace \theta_i \rbrace_{i=0}^d$
of eigenvalues for $A^{(1)}$.  For $0 \leq i,j\leq d$ the vertices $\theta_i, \theta_j$ are adjacent in $\mathcal D$ whenever
$i \not=j$ and $P(\theta_i, \theta_j)=0$.
\end{definition}
\noindent We have some comments about the graph $\mathcal D$ from Definition \ref{def:diagram}.
The polynomial $P(\lambda_1, \lambda_2)$ is quadratic in each variable, so each vertex in $\mathcal D$ is adjacent to at most two vertices in $\mathcal D$. Therefore,
each connected component of $\mathcal D$ is a path or a cycle. We are going to show that $\mathcal D$ is connected, and that $\mathcal D$
is a path.

\begin{lemma} \label{lem:Pmeaning} For distinct nonadjacent vertices $\theta_i, \theta_j$ of $\mathcal D$ we have
\begin{align} \label{eq:EBBE}
E_i^{(1)} B^{(2)} E_j^{(1)} =0, \qquad \qquad E_i^{(1)} B^{(3)} E_j^{(1)} =0.
\end{align}
\end{lemma}
\begin{proof} Immediate from Lemma  \ref{lem:EBE} and Definition \ref{def:diagram}.
\end{proof}

\begin{lemma} The graph $\mathcal D$ is connected.
\end{lemma} 
\begin{proof} Let $\Omega$ denote a nonempty subset of $\lbrace \theta_i \rbrace_{i=0}^d$ that forms a connected component of $\mathcal D$. We show that
$\Omega = \lbrace \theta_i \rbrace_{i=0}^d$. 
Define the subspace
\begin{align}
W = \sum_{\theta_i \in \Omega}  E^{(1)}_iV. \label{eq:W}
\end{align}
We have $W \not=0$ since $\Omega$ is nonempty.
We will show that $W=V$. To this end,
we will show that $W$ is an $\mathbb O_q$-submodule of $V$. We first show that $B^{(2)} W \subseteq W$.
To do this, we show that $B^{(2)}  E^{(1)}_j V \subseteq W$ for all  $\theta_j \in \Omega$. Let $\theta_j$ be given.
For
a vertex $\theta_i \not\in \Omega $ we see that $\theta_i, \theta_j$ are not adjacent, so $E_i^{(1)} B^{(2)} E_j^{(1)} =0$ by Lemma \ref{lem:Pmeaning}.
We have
\begin{align*}
B^{(2)}  E^{(1)}_j V &= I B^{(2)}  E^{(1)}_j V 
                               = \sum_{i=0}^d E^{(1)}_i B^{(2)}  E^{(1)}_j V \\
                                &=\sum_{\theta_i \in \Omega}  E^{(1)}_i B^{(2)}  E^{(1)}_j V 
                                  \subseteq \sum_{\theta_i \in \Omega}  E^{(1)}_i V =W.
\end{align*}
By this and \eqref{eq:W}, we obtain  $B^{(2)} W \subseteq W$.  The inclusion $B^{(3)} W \subseteq W$ is similarly obtained.
Let $X$ denote any one of $A^{(1)}$, $A^{(2)}$, $A^{(3)}$, $B^{(1)}$.
The map $X$  commutes with $A^{(1)}$, so each eigenspace of $A^{(1)}$ is invariant under $X$.
It follows via \eqref{eq:W} that $W$ is invariant under $X$.
We have shown that $W$ is invariant under $A^{(i)}$ and $B^{(i)}$ for all $i \in \lbrace 1,2,3\rbrace$.
Therefore, $W$
is an $\mathbb O_q$-submodule of $V$. Consequently, $W=V$ since the $\mathbb O_q$-module $V$ is irreducible.
Comparing \eqref{eq:Vsum}, \eqref{eq:W} we obtain $\Omega= \lbrace \theta_i \rbrace_{i=0}^d$, so $\mathcal D$ is connected.
\end{proof}
\noindent By our results so far, the graph $\mathcal D$ is a path or a cycle. Shortly we will show that $\mathcal D$ is a path.

\begin{lemma} \label{lem:ijk} Referring to the graph $\mathcal D$, suppose that a vertex $\theta_i$ is adjacent to two distinct vertices $\theta_j, \theta_k$. Then
\begin{align*}
\theta_j+\theta_k = (q^2+q^{-2}) \theta_i.
\end{align*}
\end{lemma}
\begin{proof} By Definition \ref{def:diagram}, $P(\theta_i, \theta_j)=0$ and $P(\theta_i, \theta_k)=0$. By this and Definition \ref{def:Poly},
\begin{align*}
0 &= \frac{P(\theta_i, \theta_j)-P(\theta_i,\theta_k)}{\theta_j-\theta_k} 
   = \theta_j + \theta_k-(q^2+q^{-2}) \theta_i.
\end{align*}
The result follows.
\end{proof}

\begin{lemma} \label{lem:Path} The graph $\mathcal D$ is a path. Label the vertices
such that $\theta_{i-1}, \theta_i$ are adjacent in $\mathcal D$ for $1 \leq i \leq d$. Then there exists a nonzero scalar $a=a^{(1)}$ in  $\mathbb F$ 
such that
\begin{align} \label{eq:eigenval}
\theta_i = a q^{d-2i} + a^{-1} q^{2i-d} \qquad \qquad (0 \leq i \leq d).
\end{align}
\end{lemma}
\begin{proof} The graph $\mathcal D$ is a path or a cycle. In either case, we may label the vertices 
such that $\theta_{i-1}, \theta_i$ are adjacent in $\mathcal D$ for $1 \leq i \leq d$.
If $\mathcal D$ is a cycle, then $d\geq 2$ and $\theta_d, \theta_0$ are adjacent in $\mathcal D$.
By Lemma \ref{lem:ijk},
\begin{align}
 \theta_{i-1} - (q^2+q^{-2}) \theta_i + \theta_{i+1} = 0 \qquad \qquad (1 \leq i \leq d-1). \label{eq:rec}
 \end{align}
 Solving the recurrence \eqref{eq:rec}, we obtain $a, \alpha \in \mathbb F$ such that 
 \begin{align} \label{eq:eig2}
 \theta_i = a  q^{d-2i} + \alpha q^{2i-d} \qquad \qquad (0 \leq i \leq d).
 \end{align}
 Using Definitions \ref{def:Poly}, \ref{def:diagram} and  \eqref{eq:eig2},
 \begin{align*}
 0 = P(\theta_{0},\theta_1) =(q^2-q^{-2})^2  (1-a  \alpha ).
 \end{align*}
 The scalar $q^2 - q^{-2}$ is nonzero because $q$ is not a root of unity.
 Therefore $a \alpha = 1$. In particular, $a \not=0$. Evaluating \eqref{eq:eig2} using $\alpha  = a^{-1}$, we obtain \eqref{eq:eigenval}.
 It remains to show that $\mathcal D$ is not
 a cycle. To this end, we assume that $\mathcal D$ is a cycle, and get a contradiction. By assumption, $d\geq 2$ and $\theta_d, \theta_0$
 are adjacent. Using Definitions \ref{def:Poly}, \ref{def:diagram} and \eqref{eq:eigenval},
 \begin{align*}
 0 = \frac{P(\theta_d, \theta_0)}{(\theta_0-\theta_{d-1})(\theta_1-\theta_d)}  = \frac{q^{d+1}-q^{-d-1}}{q^{d-1}-q^{1-d}}.
 \end{align*}
 Therefore $q^{2d+2}=1$, contradicting the assumption that $q$ is not a root of unity.  We have shown that
  $\mathcal D$ is not a cycle, so $\mathcal D$ is a path. The result follows.
\end{proof}

\noindent For the rest of this section, we label the vertices of $\mathcal D$ as in Lemma \ref{lem:Path}.

\begin{lemma} \label{lem:ineq} The scalar $a^2$ is not among $q^{2d-2}, q^{2d-4},\ldots, q^{2-2d}$.
\end{lemma} 
\begin{proof} By  \eqref{eq:eigenval} and since $\lbrace \theta_i \rbrace_{i=0}^d$ are mutually distinct.
\end{proof}

\begin{lemma} \label{lem:Pmeaning2} For $0 \leq i,j\leq d$ such that $\vert i-j\vert >1$,
\begin{align*} 
E_i^{(1)} B^{(2)} E_j^{(1)} =0, \qquad \qquad E_i^{(1)} B^{(3)} E_j^{(1)} =0.
\end{align*}
\end{lemma}
\begin{proof} The vertices $\theta_i, \theta_j$ of $\mathcal D$ are distinct and nonadjacent. The result follows in view of  Lemma \ref{lem:Pmeaning}.
\end{proof}

\begin{definition} \label{def:ti} \rm For $0 \leq i \leq d$ we define a scalar $t_i = t_i^{(1)}$ in $\mathbb F$ by
\begin{align*}
t_i = a^{2i-d} q^{2i(d-i)},
\end{align*}
where the scalar $a$ is from Lemma  \ref{lem:Path}.
\end{definition}
\noindent Observe that $t_i \not=0$ for $0 \leq i \leq d$.
\begin{lemma} \label{lem:tiFrac} For $1 \leq i \leq d$,
\begin{align*}
\frac{t_i}{t_{i-1}} = a^2 q^{2(d-2i+1)}.
\end{align*}
\end{lemma}
\begin{proof} Immediate from Definition \ref{def:ti}.
\end{proof}

\begin{lemma} \label{lem:tiFact} For $0 \leq i,j\leq d$ such that $\vert i-j\vert =1$, 
\begin{align*}
\frac{t_j}{t_i} = 1 + \frac{\theta_i - \theta_j}{q-q^{-1}} \, \frac{q \theta_i - q^{-1} \theta_j}{q^2 - q^{-2}}.
\end{align*}
\end{lemma}
\begin{proof} This is readily checked using \eqref{eq:eigenval} and Lemma \ref{lem:tiFrac}.
\end{proof}

\begin{definition}\label{def:Psi} \rm Define a map $\Psi =\Psi^{(1)}$ in ${\rm End}(V)$ by
\begin{align*}
\Psi = \sum_{i=0}^d t_i E^{(1)}_i,
\end{align*}
where the scalars $\lbrace t_i \rbrace_{i=0}^d$ are from Definition \ref{def:ti}.
\end{definition}

\begin{lemma} \label{lem:PsiInv}The map $\Psi$ is invertible. Moreover,
\begin{align*}
\Psi^{-1} = \sum_{i=0}^d t_i^{-1} E^{(1)}_i.
\end{align*}
\end{lemma}
\begin{proof} By Definition \ref{def:Psi} and the construction.
\end{proof}

\begin{theorem} \label{thm:LPsi} For $X \in \mathbb O_q$ the following holds on $V$:
\begin{align} \label{eq:LX}
L_1(X) = \Psi^{-1} X \Psi.
\end{align}
\end{theorem}
\begin{proof}  Without loss, we may assume that $X$ is a standard generator for $\mathbb O_q$.
 First assume that $X$ is one of $A_1, A_2, A_3, B_1$. Then $X$ commutes with $A_1$.
  By Theorem \ref{thm:main1},  $L_1(X)=X$.
 On $V$ we have $A_1=A^{(1)}$, so $X$ commutes with $A^{(1)}$ on $V$. For $0 \leq i \leq d$
 the element $E^{(1)}_i$ is a polynomial in $A^{(1)}$. By this and Definition \ref{def:Psi}, 
  $\Psi$ is a polynomial in $A^{(1)}$. Consequently, $X$ commutes with $\Psi$ on $V$. Thus,
  $X = \Psi^{-1}X \Psi$ on $V$. By these comments, \eqref{eq:LX} holds
   on $V$.
  Next, we assume that $X$  is one of $B_2, B_3$.  By Theorem \ref{thm:main1}, 
  \begin{align*}
  L_1(X)= X + \frac{q A_1^2 X - (q+q^{-1}) A_1 X A_1 + q^{-1} X A_1^2 }{(q-q^{-1})(q^2-q^{-2})}.
  \end{align*}
We are trying to show that   \eqref{eq:LX} holds on $V$. Since $I = \sum_{i=0}^d E^{(1)}_i$, it suffices to show that the following
holds on $V$ for $0 \leq i,j\leq d$:
\begin{align} \label{eq:needed}
E_i^{(1)} L_1(X) E_j^{(1)} = E_i^{(1)} \Psi^{-1} X \Psi E_j^{(1)}.
\end{align}
\noindent Let $i,j$ be given. On $V$ we have
\begin{align*}
E_i^{(1)} L_1(X) E_j^{(1)} &= E_i^{(1)} \biggl(   X + \frac{q \bigl( A^{(1)}\bigr)^2 X - (q+q^{-1}) A^{(1)}  X A^{(1)} + q^{-1} X \bigl( A^{(1)} \bigr)^2 }{(q-q^{-1})(q^2-q^{-2})}    \biggr) E_j^{(1)}
\\
&=E_i^{(1)} X E_j^{(1)} \biggl( 1 + \frac{ q \theta_i^2 - (q+q^{-1}) \theta_i \theta_j + q^{-1} \theta_j^2}{(q-q^{-1})(q^2-q^{-2})}     \biggr) \\
&=E_i^{(1)} X E_j^{(1)} \biggl( 1 + \frac{\theta_i - \theta_j}{q-q^{-1}} \, \frac{q\theta_i - q^{-1} \theta_j}{q^2-q^{-2}}    \biggr).
\end{align*}
\noindent By Definition \ref{def:Psi} and Lemma \ref{lem:PsiInv}, the following holds on $V$:
\begin{align*}
E_i^{(1)} \Psi^{-1} X \Psi E_j^{(1)} = E_i^{(1)} X E_j^{(1)} \frac{t_j}{t_i}.
\end{align*}
We can now easily show that \eqref{eq:needed} holds on $V$. First assume that $\vert i-j\vert >1$. Then \eqref{eq:needed} holds on $V$ because 
$E_i^{(1)} X E_j^{(1)}=0$ on $V$ by Lemma \ref{lem:Pmeaning2}.
Next assume that $i=j$. Then \eqref{eq:needed} holds on $V$ because each side is equal to $E_i^{(1)} X E_i^{(1)}$. Next assume that $\vert i-j\vert =1$.
Then \eqref{eq:needed} holds on $V$ in view of Lemma  \ref{lem:tiFact}.
We have shown that \eqref{eq:needed} holds on $V$ for $0 \leq i,j\leq d$. Therefore   \eqref{eq:LX} holds on $V$. 
By the above comments,  \eqref{eq:LX} holds on $V$ for every standard generator $X$ of $\mathbb O_q$.
The result follows.
\end{proof}

\noindent  The following example will be discussed in Section 7.

\begin{example}\label{ex:d1} \rm Assume that $d=1$. Then
\begin{align*}
&\theta_0 = a q + a^{-1} q^{-1}, \qquad \qquad \theta_1 = a q^{-1} + a^{-1} q, \qquad \qquad a^2 \not=1, \\
& A^{(1)} = \theta_0 E^{(1)}_0 + \theta_1 E^{(1)}_1 = (a-a^{-1})(q-q^{-1}) E^{(1)}_0 + (aq^{-1}+a^{-1}q) I
\end{align*}
and also
\begin{align*}
&t_0 = a^{-1}, \qquad \qquad t_1 = a, \\
\Psi &=a^{-1} E^{(1)}_0 + a E^{(1)}_1 = (a^{-1}-a) E^{(1)}_0 + a I, \\
\Psi^{-1} &= aE^{(1)}_0 + a^{-1} E^{(1)}_1 = (a-a^{-1}) E^{(1)}_0 + a^{-1} I.
\end{align*}
\end{example}

\noindent The following  definition is motivated by Theorem \ref{thm:LPsi}.
\begin{definition}\label{def:Linter} \rm The map $\Psi$ in Definition \ref{def:Psi}
is called the {\it Lusztig intertwiner}  for $L_1$ and $V$. The Lusztig interwiners for
$L_2$, $L_3$, $L^*_1$, $L^*_2$, $L^*_3$ and $V$ are similarly defined.
\end{definition}

\section{An irreducible $\mathbb O_q$-module with dimension 5}

\noindent In this section, we display an irreducible $\mathbb O_q$-module $V$ that has dimension 5. 
The actions of the standard $\mathbb O_q$-generators on $V$, and the Lusztig intertwiners for $V$, will be represented by $5 \times 5$ matrices.
\medskip

\noindent  Let ${\rm Mat}_5(\mathbb F)$ denote the algebra of $5 \times 5$ matrices that have all entries in $\mathbb F$.  Define the vector space $V=\mathbb F^5$ (column vectors).
The algebra  ${\rm Mat}_5(\mathbb F)$ acts on $V$ by left multiplication. Throughout this section, fix nonzero scalars
 $n$, $\lbrace a_i\rbrace_{i=1}^3$, $\lbrace b_i \rbrace_{i=1}^3$ in $\mathbb F$ such that
$a_i^2 \not=1$ and $b_i^2 \not=1$ for $i \in \lbrace 1,2,3\rbrace$. Using these scalars, 
we will turn $V$ into an irreducible $\mathbb O_q$-module.

\begin{definition}\label{def:AAABBB} \rm We define some matrices in ${\rm Mat}_5(\mathbb F)$:
\begin{small}
\begin{align*}
&A^{(1)} = \\
&
\begin{pmatrix}  a_1 q+a_1^{-1} q^{-1}  & 0 &(a_1-a_1^{-1})(q-q^{-1}) & (a_1-a_1^{-1})(q-q^{-1})& 0 \\
                                          0& a_1 q+a_1^{-1} q^{-1}&0&0&(a_1-a_1^{-1})(q-q^{-1}) \\
                                          0&0&a_1 q^{-1} +a_1^{-1} q&0&0\\
                                          0&0&0&a_1 q^{-1} +a_1^{-1} q&0\\
                                          0&0&0&0&a_1 q^{-1} +a_1^{-1} q
                                          \end{pmatrix},
                                           \end{align*}
 \begin{align*}
&A^{(2)} = \\
&
\begin{pmatrix}  a_2 q+a_2^{-1} q^{-1}  & (a_2-a_2^{-1})(q-q^{-1}) & 0& (a_2-a_2^{-1})(q-q^{-1})& 0 \\
                                          0& a_2 q^{-1}+a_2^{-1} q &0&0&0 \\
                                          0&0&a_2 q +a_2^{-1} q^{-1} &0&(a_2-a_2^{-1})(q-q^{-1})\\
                                          0&0&0&a_2 q^{-1} +a_2^{-1} q&0\\
                                          0&0&0&0&a_2 q^{-1} +a_2^{-1} q
                                          \end{pmatrix},
                                           \end{align*}
   \begin{align*}
&A^{(3)} = \\
&
\begin{pmatrix}  a_3 q+a_3^{-1} q^{-1}  & (a_3-a_3^{-1})(q-q^{-1}) &  (a_3-a_3^{-1})(q-q^{-1}) & 0 & 0 \\
                                          0& a_3 q^{-1}+a_3^{-1} q &0&0&0 \\
                                          0&0&a_3 q^{-1} +a_3^{-1} q &0&0\\
                                          0&0&0&a_3 q +a_3^{-1} q^{-1}&(a_3-a_3^{-1})(q-q^{-1})\\
                                          0&0&0&0&a_3 q^{-1} +a_3^{-1} q
                                          \end{pmatrix},
                                           \end{align*}
 \begin{align*}
 &B^{(1)} =\begin{pmatrix} b_1q^{-1}+b_1^{-1} q&0&0&0&0\\
                          \frac{(b_1-b_1^{-1})(q-q^{-1})}{n}&b_1 q + b_1^{-1} q^{-1}&0&0&0\\
                          0&0&b_1q^{-1}+b_1^{-1} q &0&0\\
                          0&0&0&b_1q^{-1}+b_1^{-1} q &0\\
                          0&0&\frac{(b_1-b_1^{-1})(q-q^{-1})}{n}&\frac{(b_1-b_1^{-1})(q-q^{-1})}{n}&b_1 q + b_1^{-1} q^{-1}
 \end{pmatrix},
 \end{align*}                                       
 \begin{align*}
 &B^{(2)} =\begin{pmatrix} b_2q^{-1}+b_2^{-1} q&0&0&0&0\\
                         0&b_2 q^{-1} + b_2^{-1} q&0&0&0\\
                          \frac{(b_2-b_2^{-1})(q-q^{-1})}{n}&0&b_2q+b_2^{-1} q^{-1} &0&0\\
                          0&0&0&b_2q^{-1}+b_2^{-1} q &0\\
                          0&\frac{(b_2-b_2^{-1})(q-q^{-1})}{n}&0&\frac{(b_2-b_2^{-1})(q-q^{-1})}{n}&b_2 q + b_2^{-1} q^{-1}
 \end{pmatrix},
 \end{align*}   
  \begin{align*}
 &B^{(3)} =\begin{pmatrix} b_3q^{-1}+b_3^{-1} q&0&0&0&0\\
                         0&b_3 q^{-1} + b_3^{-1} q&0&0&0\\
                          0&0&b_3q^{-1}+b_3^{-1} q &0&0\\
                          \frac{(b_3-b_3^{-1})(q-q^{-1})}{n}&0&0&b_3q+b_3^{-1} q^{-1} &0\\
                          0&\frac{(b_3-b_3^{-1})(q-q^{-1})}{n}&\frac{(b_3-b_3^{-1})(q-q^{-1})}{n}&0&b_3 q + b_3^{-1} q^{-1}
 \end{pmatrix}.
 \end{align*}                                                                    
                                           
                                       \end{small}

\end{definition}

\noindent Shortly, we will turn $V$ into an irreducible $\mathbb O_q$-module on which $A_i$ acts as $A^{(i)}$ and $B_i$ acts as $B^{(i)}$ 
 for $i \in \lbrace 1,2,3\rbrace$. In order to simplify the calculations, we give some
preliminary results.


\begin{definition} \label{def:5by5}\rm We define some matrices in ${\rm Mat}_5(\mathbb F)$:
\begin{align*}
E^{(1)} = \begin{pmatrix} 1&0&1&1&0 \\
                                         0&1&0&0& 1\\
                                         0&0&0&0&0\\
                                         0&0&0&0&0\\
                                        0 &0&0&0&0
\end{pmatrix},
\qquad \quad 
E^{*(1)} = \begin{pmatrix} 0&0&0&0&0 \\
                                         n^{-1}&1&0&0& 0\\
                                         0&0&0&0&0\\
                                         0&0&0&0&0\\
                                        0 &0&n^{-1}&n^{-1}&1
\end{pmatrix},
\end{align*}
\begin{align*}
E^{(2)} = \begin{pmatrix} 1&1&0&1&0 \\
                                         0&0&0&0& 0\\
                                         0&0&1&0&1\\
                                         0&0&0&0&0\\
                                        0 &0&0&0&0
\end{pmatrix},
\qquad \quad 
E^{*(2)} = \begin{pmatrix} 0&0&0&0&0 \\
                                         0&0&0&0& 0\\
                                         n^{-1}&0&1&0&0\\
                                         0&0&0&0&0\\
                                        0 &n^{-1}&0&n^{-1}&1
\end{pmatrix},
\end{align*}
\begin{align*}
E^{(3)} = \begin{pmatrix} 1&1&1&0&0 \\
                                         0&0&0&0& 0\\
                                         0&0&0&0&0\\
                                         0&0&0&1&1\\
                                        0 &0&0&0&0
\end{pmatrix},
\qquad \quad 
E^{*(3)} = \begin{pmatrix} 0&0&0&0&0 \\
                                         0&0&0&0& 0\\
                                         0&0&0&0&0\\
                                         n^{-1}&0&0&1&0\\
                                        0 &n^{-1}&n^{-1}&0&1
\end{pmatrix}.
\end{align*}
\end{definition}

\begin{lemma} \label{lem:ABEE} For $i \in \lbrace 1,2,3\rbrace$ we have
\begin{align}
A^{(i)} &= (a_i - a^{-1}_i)(q-q^{-1}) E^{(i)} + (a_i q^{-1} + a^{-1}_i q)I, \label{eq:AE} \\
B^{(i)} &= (b_i - b^{-1}_i)(q-q^{-1}) E^{*(i)} + (b_i q^{-1} + b^{-1}_i q)I, \label{eq:BE} \\
&\bigl(E^{(i)}\bigr)^2 = E^{(i)}, \qquad \qquad \bigl(E^{*(i)}\bigr)^2 = E^{*(i)}. \label{eq:EE}
\end{align}
\end{lemma}
\begin{proof} To get \eqref{eq:AE} and \eqref{eq:BE}, compare Definitions \ref{def:AAABBB} and \ref{def:5by5}.
To get \eqref{eq:EE}, use Definition \ref{def:5by5} and matrix multiplication.
\end{proof}

\begin{lemma} \label{lem:EE} The following relations hold in ${\rm Mat}_5(\mathbb F)$:
\begin{enumerate}
\item[\rm (i)] For $i,j \in \lbrace 1,2,3\rbrace$,
\begin{align*}
\lbrack E^{(i)}, E^{(j)} \rbrack=0, \qquad \qquad \lbrack E^{*(i)}, E^{*(j)} \rbrack=0.
\end{align*}
\item[\rm (ii)] For $i \in \lbrace 1,2,3\rbrace$,
\begin{align*}
\lbrack E^{(i)}, E^{*(i)} \rbrack=0.
\end{align*}
\item[\rm (iii)] For distinct $i,j \in \lbrace 1,2,3 \rbrace$,
\begin{align*}
n E^{(i)} E^{*(j)} E^{(i)} = E^{(i)}, \qquad \qquad n E^{*(i)} E^{(j)} E^{*(i)} = E^{*(i)}.
\end{align*}
\end{enumerate}
\end{lemma}
\begin{proof} Use Definition \ref{def:5by5} and matrix multiplication.
\end{proof} 

\begin{proposition} \label{lem:OQM} The vector space $V$ becomes an $\mathbb O_q$-module on which
 $A_i$ acts as $A^{(i)} $ and $B_i$ acts as $B^{(i)}$ for $i \in \lbrace 1,2,3\rbrace$.
\end{proposition}
\begin{proof} Using Lemmas \ref{lem:ABEE}, \ref{lem:EE} one checks that the matrices $\lbrace A^{(i)}\rbrace_{i=1}^3$, $\lbrace B^{(i)}\rbrace_{i=1}^3$ satisfy
the defining relations for $\mathbb O_q$ given in Definition  \ref{def:bbO}.
\end{proof}

\noindent Our next goal is to simultaneously diagonalize $ E^{(1)}, E^{(2)}, E^{(3)}$.
\medskip

\noindent
Define a matrix in ${\rm Mat}_5(\mathbb F)$:
\begin{align*}
T = \begin{pmatrix} 1&1&1&1&1 \\
                                         0&1&0&0& 1\\
                                         0&0&1&0&1\\
                                         0&0&0&1&1\\
                                        0 &0&0&0&1
\end{pmatrix}.
\end{align*}
Note that
\begin{align*}
T^{-1} = \begin{pmatrix} 1&-1&-1&-1&2 \\
                                         0&1&0&0& -1\\
                                         0&0&1&0&-1\\
                                         0&0&0&1&-1\\
                                        0 &0&0&0&1
\end{pmatrix}.
\end{align*}

\begin{lemma} \label{lem:Tconj} We have
\begin{align*}
T E^{(1)} T^{-1} &= {\rm diag}(1,1,0,0,0), \\
T E^{(2)} T^{-1} &= {\rm diag}(1,0,1,0,0), \\
T E^{(3)} T^{-1} &= {\rm diag}(1,0,0,1,0)
\end{align*}
and
\begin{align*}
T E^{*(1)} T^{-1} = \begin{pmatrix} n^{-1}&1-n^{-1}&0&0&0 \\
                                         n^{-1}&1-n^{-1}&0&0& 0\\
                                         0&0&n^{-1}&n^{-1}&1-2n^{-1}\\
                                         0&0&n^{-1}&n^{-1}&1-2n^{-1}\\
                                        0 &0&n^{-1}&n^{-1}&1-2n^{-1}
\end{pmatrix},
\end{align*}
\begin{align*}
T E^{*(2)} T^{-1} = \begin{pmatrix} n^{-1}&0&1-n^{-1}&0& 0   \\
                                                      0&n^{-1}&0&n^{-1}& 1-2n^{-1}  \\
                                                      n^{-1}&0&1-n^{-1}&0&0   \\
                                                      0&n^{-1}&0&n^{-1}&1-2n^{-1}  \\
                                                      0&n^{-1}&0&n^{-1}&1-2n^{-1}
\end{pmatrix},
\end{align*}
\begin{align*}
T E^{*(3)} T^{-1} = \begin{pmatrix} n^{-1}&0&0&1-n^{-1}& 0   \\
                                                      0&n^{-1}&n^{-1}&0&1-2n^{-1}  \\
                                                     0& n^{-1}&n^{-1}&0&1-2n^{-1}   \\
                                                      n^{-1}&0&0&1-n^{-1}&0 \\
                                                      0&n^{-1}&n^{-1}&0&1-2n^{-1}
\end{pmatrix}.
\end{align*}

\end{lemma}
\begin{proof}  Routine matrix multiplication.
\end{proof}

\noindent Our next goal is to simultaneously diagonalize $ E^{*(1)}, E^{*(2)}, E^{*(3)}$.
\medskip

\noindent Define a matrix in ${\rm Mat}_5(\mathbb F)$:
\begin{align*}
T^* = \begin{pmatrix} 1&0&0&0&0 \\
                                   n^{-1}&1&0&0&0 \\
                                   n^{-1}&0&1&0&0 \\
                                   n^{-1}&0&0&1& 0\\
                                   n^{-2}&n^{-1}&n^{-1}&n^{-1}&1
                                   \end{pmatrix}.
                                   \end{align*}                                 
 Note that                                 
\begin{align*}
(T^*)^{-1} = \begin{pmatrix} 1&0&0&0&0 \\
                                   -n^{-1}&1&0&0&0 \\
                                   -n^{-1}&0&1&0&0 \\
                                   -n^{-1}&0&0&1& 0\\
                                   2n^{-2}&-n^{-1}&-n^{-1}&-n^{-1}&1
                                   \end{pmatrix}.
                                   \end{align*}                                 
\begin{lemma} \label{lem:Tsconj} We have
\begin{align*}
T^* E^{(1)} (T^*)^{-1} = \begin{pmatrix}  1-2n^{-1} &0&1&1& 0\\
                                                               0 &1-n^{-1}&0&0&1 \\
                                                                n^{-1}-2n^{-2} &0&n^{-1}&n^{-1}&0 \\
                                                                 n^{-1}-2n^{-2} &0&n^{-1}&n^{-1}&0 \\
                                                                 0&n^{-1}-n^{-2}&0&0&n^{-1}
\end{pmatrix},
\end{align*}
\begin{align*}
T^* E^{(2)} (T^*)^{-1} = \begin{pmatrix}  1-2n^{-1}&1&0&1&0\\
                                                                n^{-1}-2n^{-2}&n^{-1}&0&n^{-1}&0 \\
                                                                0&0&1-n^{-1}&0&1 \\
                                                                n^{-1}-2n^{-2}&n^{-1}&0&n^{-1}&0\\
                                                                0&0&n^{-1}-n^{-2}&0&n^{-1}
     \end{pmatrix},
\end{align*}

\begin{align*}
T^* E^{(3)} (T^*)^{-1} = \begin{pmatrix} 1-2n^{-1}&1&1&0&0 \\
                                                               n^{-1}-2n^{-2}&n^{-1}&n^{-1}&0&0 \\
                                                               n^{-1}-2n^{-2}&n^{-1}&n^{-1}&0&0  \\
                                                              0&0&0&1-n^{-1}&1 \\
                                                               0&0&0&n^{-1}-n^{-2}&n^{-1}
                                                               \end{pmatrix}
\end{align*}
\noindent and
\begin{align*}
T^* E^{*(1)} (T^*)^{-1} &= {\rm diag}(0,1,0,0,1), \\
T^* E^{*(2)} (T^*)^{-1} &= {\rm diag}(0,0,1,0,1), \\
T^* E^{*(3)} (T^*)^{-1} &= {\rm diag}(0,0,0,1,1).
\end{align*}
      \end{lemma}                            
\begin{proof}  Routine matrix multiplication.
\end{proof}

\noindent Next, we list some facts about  $\lbrace A^{(i)}\rbrace_{i=1}^3$ and $\lbrace B^{(i)}\rbrace_{i=1}^3$.

\begin{lemma} \label{lem:AAABBBD} For $i \in \lbrace 1,2,3 \rbrace$ the matrices $A^{(i)}$ and $B^{(i)}$ are diagonalizable. Each has two primitive idempotents.
For $A^{(i)}$ the primitive idempotents and associated eigenvalues are:  $E^{(i)}$
associated with $a_i q + a^{-1}_i q^{-1}$,  and  $I-E^{(i)}$ associated with   $a_i q^{-1} + a^{-1}_i q$.
For $B^{(i)}$ the primitive idempotents and associated eigenvalues are: $E^{*(i)}$
associated with $b_i q + b^{-1}_i q^{-1}$,  and  $I-E^{*(i)}$ associated with   $b_i q^{-1} + b^{-1}_i q$.
\end{lemma}
\begin{proof} The matrix $A^{(i)}$ is diagonalizable by  Lemma  \ref{lem:ABEE} and Lemma \ref{lem:Tconj}. 
The matrix $B^{(i)}$ is diagonalizable by  Lemma  \ref{lem:ABEE} and Lemma \ref{lem:Tsconj}. The remaining
assertions follow from Lemma \ref{lem:ABEE}.
\end{proof}

\noindent Our next goal is to show that the $\mathbb O_q$-module $V$ is irreducible, provided that $n \not=1$ and $n \not=2$.

\begin{definition} \label{def:ingredients} \rm We define some matrices in ${\rm Mat}_5(\mathbb F)$:
\begin{align*}
&\Lambda_1 = E^{(1)} E^{(2)}, \qquad \qquad 
\Lambda_2 = E^{(1)}-\Lambda_1, \\
&\Lambda_3 = E^{(2)}-\Lambda_1, \qquad \qquad 
\Lambda_4 = E^{(3)}-\Lambda_1, \\
&\Lambda_5 = I - \Lambda_1-\Lambda_2 - \Lambda_3 - \Lambda_4.
\end{align*}
\end{definition}

\begin{lemma}\label{lem:TLT} We have
\begin{align*}
T \Lambda_1 T^{-1} &= {\rm diag}(1,0,0,0,0), \\
T \Lambda_2 T^{-1} &= {\rm diag}(0,1,0,0,0), \\
T \Lambda_3 T^{-1} &= {\rm diag}(0,0,1,0,0), \\
T \Lambda_4 T^{-1} &= {\rm diag}(0,0,0,1,0), \\
T \Lambda_5 T^{-1} &= {\rm diag}(0,0,0,0,1).
\end{align*}
\end{lemma}
\begin{proof} By Lemma \ref{lem:Tconj} and Definition \ref{def:ingredients}.
\end{proof}

\begin{lemma} \label{lem:ingredient2} For $\Lambda^*_1 = E^{*(1)} E^{*(2)}$ we have
\begin{align*}
T \Lambda^*_1 T^{-1} = \frac{1}{n^2} \begin{pmatrix} 1&n-1&n-1&n-1&(n-1)(n-2) \\
                                                                                      1&n-1&n-1&n-1&(n-1)(n-2) \\
                                                                                       1&n-1&n-1&n-1&(n-1)(n-2) \\
                                                                                        1&n-1&n-1&n-1&(n-1)(n-2) \\
                                                                                         1&n-1&n-1&n-1&(n-1)(n-2) 
                                                            \end{pmatrix}.
                                                            \end{align*}
\end{lemma}
\begin{proof} By Lemma \ref{lem:Tconj}.
\end{proof}

\begin{lemma}\label{lem:basis} Assume that $n \not=1$ and $n \not=2$. 
\begin{enumerate}
\item[\rm (i)]
The following matrices form a basis for ${\rm Mat}_5(\mathbb F)$:
\begin{align*}
\Lambda_i \Lambda^*_1 \Lambda_j, \qquad \qquad 1 \leq i,j\leq 5.
\end{align*}
\item[\rm (ii)] The algebra ${\rm Mat}_5(\mathbb F)$ is generated by 
\begin{align*}
E^{(1)}, \quad E^{(2)}, \quad E^{(3)}, \quad E^{*(1)}, \quad E^{*(2)}, \quad E^{*(3)}.
\end{align*}
\item[\rm (iii)] The algebra ${\rm Mat}_5(\mathbb F)$ is generated by 
\begin{align*}
A^{(1)}, \quad A^{(2)}, \quad A^{(3)}, \quad B^{(1)}, \quad B^{(2)}, \quad B^{(3)}.
\end{align*}
\end{enumerate}
\end{lemma}
\begin{proof} (i) It suffices to show that the following matrices form a basis for ${\rm Mat}_5(\mathbb F)$:
\begin{align}
 T \Lambda_i \Lambda^*_1 \Lambda_j T^{-1}, \qquad \qquad 1 \leq i,j\leq 5.        \label{eq:TLLLT}
\end{align}
Using Lemmas \ref{lem:TLT},  \ref{lem:ingredient2} we find that for $1 \leq i,j\leq 5$ the matrix 
$ T \Lambda_i \Lambda^*_1 \Lambda_j T^{-1}$ has $(i,j)$-entry nonzero and all other entries zero. Consequently,
the matrices  \eqref{eq:TLLLT} form a basis for ${\rm Mat}_5(\mathbb F)$. The result follows.
\\
\noindent (ii) By (i) and the construction. \\
\noindent (iii) By (ii) and \eqref{eq:AE}, \eqref{eq:BE}.
\end{proof}

\begin{proposition} The $\mathbb O_q$-module $V$ is irreducible, provided that $n \not=1$ and $n \not=2$.
\end{proposition}
\begin{proof} By Lemma \ref{lem:basis}(iii) and since $V$ is irreducible as a module for ${\rm Mat}_5(\mathbb F)$.
\end{proof}

\noindent Next, we describe the Lusztig intertwiners for $V$.

\begin{proposition} \label{prop:Lint} The following hold for $i \in \lbrace 1,2,3\rbrace$.
\begin{enumerate}
\item[\rm (i)]  The Lusztig intertwiner for $L_i$ and $V$ is 
\begin{align*}
\Psi^{(i)} = (a_i^{-1}-a_i) E^{(i)} + a_i I.
\end{align*}
\item[\rm (ii)] The Lusztig intertwiner for $L^*_i$ and $V$ is 
\begin{align*}
\Psi^{*(i)} = (b_i^{-1}-b_i) E^{*(i)} + b_i I.
\end{align*}
\end{enumerate}
\end{proposition}
\begin{proof} The Lusztig intertwiner for $L_1$ and $V$ is obtained from Example \ref{ex:d1} and Lemma \ref{lem:AAABBBD}.
The other Lusztig intertwiners for $V$ are similarly obtained.
\end{proof}



\bigskip


\noindent Paul Terwilliger \hfil\break
\noindent Department of Mathematics \hfil\break
\noindent University of Wisconsin \hfil\break
\noindent 480 Lincoln Drive \hfil\break
\noindent Madison, WI 53706-1388 USA \hfil\break
\noindent email: {\tt terwilli@math.wisc.edu }\hfil\break

\section{Statements and Declarations}

\noindent {\bf Funding}: The author declares that no funds, grants, or other support were received during the preparation of this manuscript.
\medskip

\noindent  {\bf Competing interests}:  The author  has no relevant financial or non-financial interests to disclose.
\medskip

\noindent {\bf Data availability}: All data generated or analyzed during this study are included in this published article.


\begin{thebibliography}{10}




\bibitem{bbit}
E.~Bannai, Et.~Bannai, T.~Ito, R.~Tanaka.
\newblock
{\em Algebraic Combinatorics.}
\newblock De Gruyter Series in Discrete Math and Applications 5.
De Gruyter, 2021.      \\
https://doi.org/10.1515/9783110630251
 


\bibitem{qOnsager1}
P.~ Baseilhac. 
\newblock
An integrable structure related with tridiagonal algebras.
\newblock{\em Nuclear Phys.
B} 705 (2005), 605--619; {\tt arXiv:math-ph/0408025}.



\bibitem{qOnsager2}
P.~Baseilhac.
\newblock
Deformed Dolan-Grady relations in quantum integrable models.
\newblock{\em  Nuclear
Phys. B} 709 (2005), 491--521; {\tt arXiv:hep-th/0404149}.






 \bibitem{bas8}
 P.~Baseilhac and S.~Belliard.
  \newblock
   Generalized $q$-Onsager algebras and boundary affine Toda field theories.
    \newblock {\em
     Lett. Math. Phys.}  93  (2010),  213--228; {\tt arXiv:0906.1215}.


\bibitem{basXXZ}
P.~Baseilhac and S.~Belliard.
\newblock
The half-infinite XXZ chain in Onsager's approach.
\newblock
{\em
Nuclear Phys. B} 873 (2013), 550--584;  {\tt arXiv:1211.6304}.



\bibitem{basBel}
P.~Baseilhac and S.~Belliard.
\newblock An attractive basis for the $q$-Onsager algebra;
      {\tt arXiv:1704.02950}.


\bibitem{BK05}
P.~Baseilhac and K.~Koizumi.
\newblock A new (in)finite dimensional algebra for
quantum integrable models.
\newblock {\em 
 Nuclear Phys. B}  720  (2005), 325--347;  {\tt arXiv:math-ph/0503036}.







\bibitem{bas4}
 P.~Baseilhac and K.~Koizumi.
   \newblock
    A deformed analogue of Onsager's symmetry
      in the
       $XXZ$ open spin chain.
       \newblock {\em
        J. Stat. Mech. Theory Exp.}  2005,  no. 10, P10005, 15 pp. (electronic);  {\tt arXiv:hep-th/0507053}.



  \bibitem{basKoi}
    P.~Baseilhac and K.~Koizumi.
       \newblock
    Exact spectrum of the $XXZ$ open spin chain from the
          $q$-Onsager algebra representation theory.
          \newblock {\em  J. Stat. Mech. Theory Exp.}
       2007,  no. 9, P09006, 27 pp. (electronic);
    {\tt arXiv:hep-th/0703106}.










\bibitem{kolb}
P.~Baseilhac and S.~Kolb. 
\newblock Braid group action and root vectors for the $q$-Onsager algebra.
\newblock{\em 
Transform. Groups} 25 (2020), 363–389; {\tt arXiv:1706.08747}.


\bibitem{basnc}
 P.~Baseilhac and K.~Shigechi.
  \newblock
   A new current algebra and the reflection equation.
    \newblock{\em
     Lett. Math. Phys. }
       92
        (2010),   47--65; {\tt arXiv:0906.1482v2}.




\bibitem{goff}
O.~Goff.
\newblock
The $q$-Onsager algebra and the quantum torus.
\newblock {\em J. Combin. Theory Ser. A} 208  (2024), 105939;
{\tt arXiv:2304.09326}.



\bibitem{itoTD}
T.~Ito.
\newblock
 TD-pairs and the $q$-Onsager algebra.
 \newblock{\em 
Sugaku Expositions} 32 (2019),  205--232.





\bibitem{someAlg}
T.~Ito, K.~Tanabe, P.~Terwilliger.
\newblock Some algebra related to $P$- and $Q$-polynomial association schemes.
\newblock{\em
Codes and Association Schemes (Piscataway NJ, 1999), 167--192, DIMACS
Ser. Discrete Math. Theoret. Comput. Sci.}
{56}, Amer. Math. Soc., Providence RI 2001;
{\tt arXiv:math.CO/0406556}.


\bibitem{TDqRacah}
T.~Ito, P.~Terwilliger.
\newblock
Tridiagonal pairs of $q$-Racah type.
\newblock{\em 
J. Algebra}  322 (2009), 68--93; {\tt  arXiv:0807.0271}.
 

\bibitem{augIto}
T.~Ito, P.~Terwilliger.
\newblock
The augmented tridiagonal algebra.
\newblock{\em  Kyushu J. Math.} 64 (2010), 81--144; {\tt arXiv:0904.2889}.


 
 



\bibitem{kolb2}
S. Kolb.
\newblock Quantum symmetric Kac-Moody pairs.
\newblock{\em  Adv. Math.} 267 (2014), 395--469; {\tt arXiv:1207.6036}.


\bibitem{lemarthe}
G.~Lemarthe, P.~Baseilhac, A.~Gainutdinov.
\newblock
Fused $K$-operators and the $q$-Onsager algebra.
\newblock Preprint;
{\tt arXiv:2301.00781}.  


\bibitem{iHall}
M.~Lu, S.~Ruan, W. Wang.
\newblock
 $\iota$Hall algebra of the projective line and $q$-Onsager algebra
\newblock{\em 
Trans. Amer. Math. Soc.} 376 (2023), 1475--1505;  {\tt arXiv:2010.00646}.

\bibitem{wang2}
M.~Lu and W.~Wang.
\newblock
A Drinfeld type presentation of affine $\iota$quantum groups I: Split ADE type.
\newblock{\em 
Adv. Math.} 393 (2021), Paper No. 108111, 46 pp; 	{\tt arXiv:2009.04542}.

\bibitem{Onsager}
L.~Onsager.
\newblock
Crystal statistics. I.
A two-dimensional model with an
order-disorder transition.
\newblock{\em Phys. Rev. (2)}  { 65} (1944),
117--149.


\bibitem{tSub3} 
P.~Terwilliger. 
\newblock The subconstituent algebra of
an association scheme III. 
\newblock{ \em
J. Algebraic Combin.}
{ 2} (1993), 177--210.  




 \bibitem{qSerre}
  P.~Terwilliger.
    \newblock Two relations that generalize the $q$-Serre relations and the
      Dolan-Grady relations. In
        \newblock {\em  Physics and
	  Combinatorics 1999 (Nagoya)}, 377--398, World Scientific Publishing,
	     River Edge, NJ, 2001;
	     {\tt arXiv:math.QA/0307016}.




\bibitem{madrid}
P.~Terwilliger.
\newblock
An algebraic approach to the Askey scheme of orthogonal polynomials.
\newblock
Lecture Notes in Math., 1883
Springer-Verlag, Berlin, 2006, 255--330;    {\tt arXiv:math/0408390}.

\bibitem{LusTer}
P.~Terwilliger.
\newblock
The Lusztig automorphism of the $q$-Onsager algebra.
\newblock{\em 
J. Algebra} 506 (2018), 56--75; {\tt arXiv:1706.05546}.





\bibitem{qOnsUAW}
P.~Terwilliger.
\newblock
 The $q$-Onsager algebra and the universal Askey-Wilson algebra.
 \newblock{\em 
SIGMA Symmetry Integrability Geom. Methods Appl.} 14 (2018), Paper No. 044, 18 pp.; {\tt arXiv:1801.06083}. 


\bibitem{TDqtet}
P.~Terwilliger.
\newblock
Tridiagonal pairs of $q$-Racah type and the $q$-tetrahedron algebra.
\newblock{\em
J. Pure Appl. Algebra} 225 (2021), no. 8, Paper No. 106632, 33 pp.; {\tt arXiv:2006.02511}.

\bibitem{congTer}
P.~Terwilliger.
\newblock
 A conjecture concerning the $q$-Onsager algebra.
 \newblock{\em  Nuclear Phys.} B 966 (2021), Paper No. 115391, 26 pp.;
 {\tt arXiv:2101.09860}.
\bibitem{ACETer}
P.~Terwilliger.
\newblock
The alternating central extension of the $q$-Onsager algebra.
\newblock{\em 
Comm. Math. Phys. } 387  (2021),  1771--1819;
{\tt arXiv:2103.03028}.
\bibitem{TerACE}
P.~Terwilliger.
\newblock
The $q$-Onsager algebra and its alternating central extension. 
\newblock{\em 
 Nuclear Phys.} B  975  (2022), Paper No. 115662, 37 pp.;
{\tt arXiv:2106.14041}.


\bibitem{comTer}
P.~Terwilliger.
\newblock
The compact presentation for the alternating central extension of the $q$-Onsager algebra.
\newblock{\em 
J. Pure Appl. Algebra} 227 (2023), no. 11, Paper No. 107408, 22 pp.;
 {\tt arXiv:2103.11229}. 


\bibitem{twist}
P.~Terwilliger.
\newblock
Twisting finite-dimensional modules for the $q$-Onsager algebra $\mathcal O_q$ via the Lusztig automorphism.
\newblock{\em 
Ramanujan J.}  61 (2023), 175--202; {\tt arXiv:2005.00457}.


\bibitem{S3}
P.~Terwilliger.
\newblock The $S_3$-symmetric tridiagonal algebra. 
\newblock Preprint; 
   {\tt arXiv:2407.00551}.			

\end{thebibliography}
\end{document}